\documentclass[a4paper,12pt]{amsart}
\usepackage{a4wide}
\usepackage{latexsym}
\usepackage{amsmath,amssymb,amsxtra,amscd,amsthm}
\usepackage{bbm}
\usepackage{rotating}
\usepackage{color}
\usepackage{hyperref}

\numberwithin{equation}{section}
\allowdisplaybreaks

\usepackage{graphicx}
\usepackage{fancyhdr}

\numberwithin{equation}{section} 
\numberwithin{figure}{section} 
\theoremstyle{plain}
\newtheorem{thm}{Theorem}[section]

\newtheorem{prop}[thm]{Proposition}

\newtheorem{cor}[thm]{Corollary}
\newtheorem{lem}[thm]{Lemma}

\theoremstyle{definition}
\newtheorem{dfn}[thm]{Definition}

\begin{document}
	
	\title{Gauss-Manin Lie algebra of mirror elliptic K3 surfaces}
	
	\author{Murad Alim}
	\address{FB Mathematik, Universit\"at Hamburg, Bundesstr. 55, 20146, Hamburg, Germany}
	\email{murad.alim@uni-hamburg.de}
	
	\author{Martin Vogrin}
	\address{FB Mathematik, Universit\"at Hamburg, Bundesstr. 55, 20146, Hamburg, Germany}
	\email{martin.vogrin@uni-hamburg.de}
	

	\maketitle
	\begin{abstract}
		We study mirror symmetry of families of elliptic K3 surfaces with elliptic fibers of type $E_6,~E_7$ and $E_8$. We consider a moduli space $\mathsf{T}$ of the mirror K3 surfaces enhanced with the choice of differential forms. We show that coordinates on $\mathsf{T}$ are given by the ring of quasi modular forms in two variables, with modular groups adapted to the fiber type.  We furthermore introduce an algebraic group $\mathsf{G}$ which acts on $\mathsf{T}$ from the right and construct its Lie algebra $\mathrm{Lie}(\mathsf{G})$. We prove that the extended Lie algebra generated by $\mathrm{Lie}(\mathsf{G})$ together with modular vector fields on $\mathsf{T}$ is isomorphic to $\mathrm{sl}_2(\mathbb{C})\oplus\mathrm{sl}_2(\mathbb{C})$.
	\end{abstract}
	\section{Introduction}
	
	Mirror symmetry was discovered within string theory identifying the complex geometry of a family of Calabi-Yau (CY) threefolds with the symplectic geometry of a mirror family of CY threefolds. This identification was formulated mathematically as an isomorphism of variations of Hodge structure (VHS), see e.g.~\cite{cox:1999} and references therein. Associated to the holomorphic Gauss-Manin connection there is a non-holomorphic $\mathcal{C}^{\infty}$ flat Gauss-Manin connection which gives the underlying moduli spaces the structure of a special K\"ahler manifold. This flat connection was given by Strominger in Ref.~\cite{Strominger:1990}, see also Ref.~\cite{Ceresole:1993} and references therein for a review as well as Freed's \cite{Freed:1997} formulation of special K\"ahler geometry.
	
	Mirror symmetry has triggered many exciting developments in mathematics and has been extended to CY manifolds in other dimensions as well as beyond CY, see Ref.~\cite{Hori:2003}. A rich extension of special geometry is given by Cecotti and Vafa's $tt^*$-geometry \cite{Cecotti:1991}, formulated mathematically in Refs.~\cite{Dubrovin:1992, Hertling:2002}. In this work we will study mirror symmetry for CY twofolds given by lattice polarized K3 surfaces as introduced by Dolgachev \cite{Dolgachev:1996}. The mirror constructions that we will use are the ones based on toric varieties as put forward by Batyrev and Borisov in Refs.~\cite{Batyrev:1994,Batyrev:1996}. A formulation of the special K\"ahler geometry in this case was given in Ref.~\cite{Alim:2014} based on $tt^*$ geometry.
	
	In the realm of Calabi-Yau threefolds, the wide mathematical interest in mirror symmetry was triggered by the predictions for the enumerative geometry of the genus 0 and arbitrary degree curves of the quintic threefold of Candelas et al. \cite{Candelas:1991}. A formulation of mirror symmetry at higher genus was given by Bershadsky, Ceccotti, Ooguri and Vafa \cite{Bershadsky:1994}, providing a recursive procedure to obtain the generating functions of higher genus curve counts using the lower genus ones. In Ref.~\cite{Yamaguchi:2004}, Yamaguchi and Yau obtained a differential ring of finitely many special functions on the moduli space of the mirror quintic and related geometries with one dimensional moduli spaces and proved that the higher genus generating functions are polynomial in the generators, this result was generalized to arbitrary CY threefolds in Ref.~\cite{Alim:2007}. The analogous differential rings for the elliptic curve were put forward in Ref.~\cite{Hosono:2008} as well as for K3 manifolds in Ref.~\cite{Alim:2014}. The differential rings of Refs.~\cite{Yamaguchi:2004,Alim:2007} were constructed from special K\"ahler geometry as well as from the Gauss-Manin connection in Ref.~\cite{Zhou:2013} and operations such as holomorphic limit and completion, analogous to operations on the ring of quasi-modular forms, were realized.
	
	In a different strand of research, the work of Movasati \cite{Movasati:20113} provided a complex geometric setting for quasi modular forms by considering the Gauss-Manin connection on the complex three dimensional moduli space $\mathsf{T}$ of elliptic curves enhanced with choices of two differential forms. Out of the vector fields on $\mathsf{T}$, an $\mathrm{sl}_2(\mathbb{C})$ Lie algebra was constructed. One of the generators of this Lie algebra is identified with the modular vector field which corresponds to Ramanujan's relations between derivatives of quasi-modular forms. The analogous procedure was continued by Movasati for the mirror quintic \cite{Movasati:20111} obtaining a holomorphic version of the differential ring of Yamaguchi and Yau. In Ref.~\cite{Alim:2016}, a synthesis of Movasati's approach and the differential rings of the special geometry of CY threefolds \cite{Yamaguchi:2004,Alim:2007} was given, putting forward a moduli space $\mathsf{T}$ of CY threefolds enhanced with choices of differential forms, together with the analogous Lie algebra $\mathfrak{G}$ which we will refer to in this work as the Gauss-Manin Lie algebra, this program was called Gauss Manin Connection in Disguies (GMCD) by the authors \cite{Alim:2016}, see also \cite{Movasati:2017} for more details. Constructions of the analogous moduli spaces $\mathsf{T}$ were given in Ref.~\cite{Alim:2014} in relation to $tt^*$-geometry of CY $d$-folds, $d=1,2,3$, as well as for Dwork families of CY $d$-folds in Ref.~\cite{Movasati:2016}, an $\mathrm{sl}_2(\mathbb{C})$ sub Lie algebra for these families was studied in Ref.~\cite{Nikdelan:2017}. It was moreover used in Ref.~\cite{Haghighat:2015} for the study of modularity of elliptic fibrations as well as for the study of Noether Lefschetz loci in Ref.~\cite{Movasati:2014}. 
	
	The focus of study of the current work is elliptically fibered lattice polarized K3 surfaces. The GMCD has been studied for particular families of lattice polarized K3's in Refs.~\cite{Doran:2014,Movasati:2016,Nikdelan:2017}. A description of the expected moduli space $\mathsf{T}$ for arbitrary lattice polarized K3's has been given in Ref.~\cite{Alim:2014}, based on $tt^*$ geometry. In this paper we will give an alternative, holomorphic construction of the moduli space $\mathsf{T}$ attached to the elliptically fibered geometries. An elliptic K3 surface is a K3 surface $X$, together with a surjective morphism $\pi:X\rightarrow\mathbb{P}^1$, such that the general fiber is an elliptic curve. The singular fibers of $\pi$ were classified in \cite{Kodaira:1963}. A particularly interesting subset of these, which are known to exhibit modular properties are projective K3 surfaces with singular fibers of type $E_6,~E_7$ and $E_8$, corresponding to elliptic singularities of the same type. In Refs.~\cite{Lian:19951,Lian:19952} Lian and Yau proved that a fundamental solution of the Picard-Fuchs equations for the mirrors of these K3 surfaces factorizes as a product of two modular forms for congruence subgroups of $SL_2(\mathbb{Z})$. The modular subgroup is given by the respective monodromy group of an elliptic fiber of the mirror. Moreover, the authors discovered an intricate relationship between the mirror map for these families and the McKay-Thompson series \cite{Lian:1996} (see also Ref.~\cite{Doran:2000}). An orthogonal approach to connect K3 periods to quasi modular forms was taken in Ref.~\cite{Yang:2007}, where the authors started with certain quasi-modular forms in two parameters and constructed K3 surfaces for which these modular forms are realized as classical periods.

	In this work, we will construct the moduli space $\mathsf{T}$ from the data of the holomorphic Gauss-Manin connection of the middle dimensional cohomology of the mirrors of the elliptically fibered K3 manifolds. We show that $\mathsf{T}$ is 6-dimensional in accordance with its general construction based on special geometry of Ref.~\cite{Alim:2014}. Away from the discriminant locus $\mathsf{T}$ is a locally ringed space with local rings $\mathcal{O}_\mathsf{T}$.  We will show that there is an isomorphism
	\begin{equation}
		\mathcal{O}_{\mathsf{T}}\cong \widetilde{\mathsf{M}}(\Gamma_0(N)\times\Gamma_0(N)),\label{eq:iso}
	\end{equation}
	between the local ring $\mathcal{O}_{\mathsf{T}}$ and the graded ring of quasi-modular forms of the modular subgroup $\Gamma_0(N)$ in two variables. The level $N$ of the congruence subgroup is determined by the type of elliptic fiber, as explained in Section \ref{sec:vhs}. Moreover we will construct the Gauss-Manin Lie algebra $\mathfrak{G}$ attached to $\mathsf{T}$ and prove in Theorem \ref{thm:mainthm} that there is an isomorphism:
	\begin{equation}
		\mathfrak{G}\cong \mathrm{sl}_2(\mathbb{C})\oplus \mathrm{sl}_2(\mathbb{C}).
	\end{equation}

	This paper is organized as follows; In section \ref{sec:ms} we review the setting of this work including the description of the moduli space $\mathsf{T}$ and mirror symmetry for projective elliptic K3 surfaces following to large extent existing literature \cite{Hosono:2000,Candelas:1993,Candelas:1994,Hosono:1993,Hosono:1994}. We construct the variation of Hodge structure for these, following \cite{cox:1999} in the first part. In section \ref{sec:vhs} we introduce the notion of algebraic variation of Hodge structure and we construct explicit coordinates on enlarged moduli spaces $\mathsf{T}$. We show that the local ring $\mathcal{O}_{\mathsf{T}}$ is isomorphic to the ring of quasi-modular forms in two variables. In the fourth section we construct the algebraic group $\mathsf{G}$ on $\mathsf{T}$ and compute its Lie algebra. We show that the Lie algebra $\mathrm{Lie}(\mathsf{G})$ extended by the Ramanujan vector fields $\mathsf{R}_a$, introduced in \eqref{eq:modvf}, is isomorphic to $\mathrm{sl}_2(\mathbb{C})\oplus\mathrm{sl}_2(\mathbb{C})$.
	
	\subsection*{Acknowledgements}
	We would like to thank Florian Beck for comments on the manuscript. M.V. would like to thank the Yau Mathematical Sciences Center at Tsinghua University, where part of this work was carried out, for hospitality and Babak Haghighat and Si Li for helpful discussions. This research is supported by DFG Emmy-Noether grant on ”Building blocks of physical theories from the geometry of quantization and BPS states”, number AL 1407/2-1.

	\thanks
	
	\section{Mirror symmetry for elliptic K3 surfaces}\label{sec:ms}
	
	\subsection{Lattice polarized K3 surfaces}
	In this section we give a short review of mirror symmetry for lattice polarized projective K3 surfaces following to a large extent \cite{Dolgachev:1996,Dolgachev:2013,Hosono:2000}. A lattice polarized K3 surface is defined by an even, non-degenerate lattice $M$ of signature $(1,19-m)$ that admits a primitive embedding into the K3 lattice $\Lambda_{K3}=E_8(-1)\oplus E_8(-1)\oplus H^{\oplus 3}$ where $H$ represents the rank two hyperbolic lattice. An $M$-polarized K3 surface is a K3 surface $X$ whose Picard lattice $\rm{Pic}(X)$ is given by $M$. The orthogonal complement of $M$ in $\Lambda_{K3}$ gives the transcendental lattice of the K3 surface and, up to a factor of $H$, the lattice of the mirror K3. Consider a mirror pair $X_\Delta,X_{\Delta^\circ}$ of projective K3 surfaces described by dual three-dimensional reflexive, integral polytopes $(\Delta,\Delta^\circ)$, as introduced in \cite{Batyrev:1994,Batyrev:1996}, and denote by $\mathbb{P}_\Delta$ and $\mathbb{P}_{\Delta^\circ}$ the ambient toric varieties of $X_\Delta$ and $X_{\Delta^\circ}$ respectively. The polarization of the two surfaces is given by the pull-back of toric divisors, together with the divisors that arise from possible splitting of the simple divisors intersected with the hypersurfaces into several irreducible components. A toric Picard lattice is defined as $\mathrm{Pic}_{tor}(\Delta)=\iota^*A^1(X_\Delta)$, where $\iota: X_\Delta\rightarrow \mathbb{P}_\Delta$ denotes the embedding of the K3 surface into the toric variety and $A^1(\mathbb{P}_\Delta)$ is the first Chow group of the toric variety. A mirror of a lattice polarized K3 surface $(X_\Delta,\mathrm{Pic}_{tor}(\Delta))$ is a lattice polarized K3 surface $(X_{\Delta^\circ},\mathrm{Pic}_{cor}(\Delta^\circ))$, where $\mathrm{Pic}_{cor}(\Delta^\circ)$ is the orthogonal complement of $\mathrm{Pic}_{tor}(\Delta^\circ)$ in $H^2(X,\mathbb{Z})$, see e. g. \cite{Rohsiepe:2004}.

	\subsection{Realization of elliptic K3 surfaces as hypersurfaces in weighted projective spaces}
	Elliptic K3 surfaces with singular fibres of types $E_6,~E_7$ and $E_8$ can be obtained as hypersurfaces of degrees $d_6=6,~d_7=8$ and $d_8=12$ in weighted projective spaces $\mathbb{P}(2,2,1,1)$, $\mathbb{P}(4,2,1,1)$ and $\mathbb{P}(6,4,1,1)$. Explicitly, they are given as the zero loci of Fermat polynomials of the form
	\begin{equation}\label{eq:eq1}
		\left\{f(x)=x_1^{d/w_1}+x_2^{d/w_2}+x_3^d+x_4^d=0\right\}\subset\mathbb{P}(w_1,w_2,1,1),
	\end{equation}
	where $x_i$ denote the homogeneous coordinates in $\mathbb{P}(w_1,w_2,1,1)$. In all three cases there is a singular locus along $x_3=x_4=0$ of the torus action resulting in the singular curve $\mathcal{C}:x_1^{d/w_1}+x_2^{d/w_2}=0$. The singularity in the ambient space can be resolved by introducing a linear relation $x_4=\lambda x_3$. This defines an exceptional divisor $E$, which is a ruled surface over the curve $\mathcal{C}$. The resulting geometry is a K3 surface which is a double cover of the associated elliptic curves by the map $(x_1,x_2,x_3,x_4)\mapsto (\lambda=x_3/x_4;x_1,x_2,y_3=x_3^2)$, branched over the elliptic curve. Elliptic fibers of type $E_6,~E_7$ and $E_8$ at a general point are given by hypersurfaces $\mathbb{P}(1,1,1)[3],~ \mathbb{P}(2,1,1)[4]$ and $\mathbb{P}(3,2,1)[6]$, where the degree of the hypersurface is indicated in the square brackets. The monodromy groups of these elliptic curves are genus zero congruence subgroups of $SL_2(\mathbb{Z})$. We define them in Appendix A and introduce quasi-modular forms for them. The congruence subgroups are $\Gamma_0(3)$ for elliptic curve of type $E_6$, $\Gamma_0(2)$ for elliptic curve of type $E_7$ and $\Gamma_0(1)$ for elliptic curve of type $E_8$. The polytopes $\Delta=\{D_i\}, i\in\{0,1,2,3,4,5\}$ for elliptic K3 surfaces in \eqref{eq:eq1} are given in Table \ref{tab:tab1}. The vectors $l^{(i)}$ of linear relations between the polytopes generate the Mori cone in the secondary fan of $\mathbb{P}_{\Delta^\circ}$.
	\begin{table}
		\centering
		\begin{tabular}{c|c c c c| c c}
			&&&&&$l^{(1)}$&$l^{(2)}$\\
			$D_0$&1&0&0&0&$-(w_1+w_2)/2-1$&0\\
			$D_1$&1&0&0&1&$w_1/2$&0\\
			$D_2$&1&0&1&0&$w_2/2$&0\\
			$D_3$&1&1&$w_2/2$&$w_1/2$&0&1\\
			$D_4$&1&-1&$w_2/2$&$w_1/2$&0&1\\
			$D_5$&1&0&$w_2/2$&$w_1/2$&1&-2
		\end{tabular}
		\caption{}\label{tab:tab1}
	\end{table}
	The intersection numbers can be computed from the toric data. Let $L$ be the linear system generated by degree one polynomials $x_3$ and $x_4$, and let $H$ be the linear system generated by degree 2 polynomials $x_3^2,~x_3x_4,~x_4^2,~x_1$ and $x_2$. It is straightforward to compute the intersection numbers $C_{LL}=L\cdot L=0,~C_{LH}=C_{HL}=H\cdot L=2d/(w_1 w_2),~C_{HH}=H\cdot H= 4d/(w_1 w_2)$.
	
	The mirror variety $X_{\Delta^\circ}$ is constructed by the Batyrev-Borisov construction \cite{Batyrev:1996} as a resolution of the quotient
	\begin{equation}\label{eq:mirror}
		\left\{f_{\phi,\psi}(y)=y_1^{d/w_1}+y_2^{d/w_2}+y_3^d+y_4^d-d\psi y_1y_2y_3y_4-2\phi y_3^{d/2}y_4^{d/2}=0\right\}\slash G,
	\end{equation}
	in $\mathbb{P}(w_1,w_2,1,1)$ with coordinates $y_1, y_2, y_3, y_4$ and $d$ is the degree of $f_{\phi,\psi}(y)$. The discrete groups $G$ are $(\mathbb{Z}/3\mathbb{Z})\times(\mathbb{Z}/2\mathbb{Z})$, $(\mathbb{Z}/4\mathbb{Z})\times(\mathbb{Z}/2\mathbb{Z})$ and $(\mathbb{Z}/3\mathbb{Z})\times(\mathbb{Z}/3\mathbb{Z})$ for the three K3 surfaces in the usual order. The locus $f_{\psi,\phi}(y)=0$ defines a family $\mathcal{X}$ of lattice polarized projective K3 surfaces over the moduli space $\mathsf{B}$ of dimension 2 for which $\psi$ and $\phi$ provide a local coordinate chart. The polarization is given, up to a a rank two hyperbolic lattice, by the orthogonal complement of the toric divisors defined by $\{D_i\}$ in the K3 lattice $\Lambda_{K3}$. For later purposes it will be useful to write down a general polynomial of the form \eqref{eq:mirror}
	\begin{equation}\label{eq:mirror2}
		f_{a_1,\ldots,a_5}(y)=a_1y_1^{d/w_1}+a_2y_2^{d/w_2}+a_3y_3^d+a_4y_4^d+a_0 y_1y_2y_3y_4+a_5 y_3^{d/2}y_4^{d/2}.
	\end{equation}
	It is equivalent to the form \eqref{eq:mirror} by a projective transformation of $y_i$. We define the GKZ coordinates
	\begin{equation}
		z_1=\frac{a_1^{w_1/2}a_2^{w_2/2}a_5}{a_0^{d/2}},\qquad z_2=\frac{a_3a_4}{a_5^2}.
	\end{equation}
	The relations between $(\psi,\phi)$ and $(z_1,z_2)$ are $-d\psi=z_1^{-2/d}z_2$ and $-2\phi=z_2^{-1/2}$. We will denote the homogeneous polynomial $f_{\psi,\phi}$ in \eqref{eq:mirror} by $f_{z_1,z_2}$ to highlight the dependence on $z_1$ and $z_2$.
	
	\subsection{Variation of Hodge structure and Picard-Fuchs equations}
	Let $\mathsf{B}$ be a complex manifold and let $\pi:\mathcal{X}\rightarrow\mathsf{B}$ be a family of K3 surfaces. The vector bundle $\mathcal{H}^2_{\mathrm{dR}}(\mathcal{X})=R^2\pi_*\mathbb{C}\otimes\mathcal{O}_\mathsf{B}$ carries the Gauss-Manin connection $\nabla:\mathcal{H}^2_{\mathrm{dR}}(\mathcal{X})\rightarrow\mathcal{H}^2_{\mathrm{dR}}(\mathcal{X})\otimes_{\mathcal{O}_\mathsf{B}}\Omega^1_\mathsf{B}$ defined by the action on the locally constant subsheaf $R^2\pi_*\mathbb{C}$ by
	\begin{equation}
		\nabla(s\otimes f)=s\otimes df, 
	\end{equation}
	for $s\in R^2\pi_*\mathbb{C}, f\in\mathcal{O}_\mathsf{B}$, where $\mathcal{O}_{\mathsf{B}}$ denotes the $\mathbb{C}$-algebra of regular functions on $\mathsf{B}$ and by $\Omega^1_{\mathsf{B}}$ we denote the $\mathcal{O}_{\mathsf{B}}$-module of differential 1-forms on $\mathsf{B}$. The Hodge filtration $F^\bullet(X_b)=\{F^p(X_b)\}_{p=0,1,2}=\bigoplus_{a\geq p}H^{a,2-a}(X_b)$ for each fiber specifies the Hodge bundle $\mathcal{F}^\bullet$ of the family $\mathcal{X}$. The Hodge filtration $F^\bullet(X_b)$ varies holomorphically over the base $\mathsf{B}$ and $\nabla$ satisfies Griffiths' transversality
	\begin{equation}
		\nabla \mathcal{F}^p\subset \mathcal{F}^{p-1}\otimes_{\mathcal{O}_{\mathsf{B}}}\Omega^1_{\mathsf{B}}.
	\end{equation}
	We say that a family $\mathcal{X}$ of K3 surfaces is polarized by a lattice $M$ if each fiber of $\mathcal{X}$ is polarized by $M$. The image of the polarization $\iota: M\rightarrow \mathcal{H}^2_{\mathrm{dR}}(\mathcal{X})$ consists of constant sections of the Gauss-Manin connection. We will denote by $\nabla:\mathcal{H}^2_{\mathrm{dR}}(\mathcal{X})_{\iota}\rightarrow\mathcal{H}^2_{\mathrm{dR}}(\mathcal{X})_{\iota}\otimes_{\mathcal{O}_\mathsf{B}}\Omega^1_\mathsf{B}$ the induced connection on the quotient $\mathcal{H}^2_{\mathrm{dR}}(\mathcal{X})_{\iota}=\mathcal{H}^2_{\mathrm{dR}}(\mathcal{X})/\iota(M)$. Furthermore, we will denote by $\mathcal{F}^\bullet_{\iota}$ a filtration on $\mathcal{H}^2_{\mathrm{dR}}(\mathcal{X})_{\iota}$ induced by $F^\bullet$. We say that a basis $\omega=(\omega_1,\ldots,\omega_{m+2})$ of $H^2_{\rm dR}(X_b)_{\iota}$ is compatible with its Hodge filtration if $\omega_1\in F^2,~\omega_2,\ldots,\omega_{m+1}\in F^1\setminus F^2$ and $\omega_{m+2}\in F^0\setminus F^1$.
	
	For families of projective K3 surfaces in \eqref{eq:mirror} the variation of Hodge structure can be constructed with the Griffiths-Dwork method, as reviewed below. Define a holomorphic $3$-form on $\mathbb{P}(w_1,w_2,w_3,w_4)$ by
	\begin{equation}
		\Omega_{\mathbb{P}(w_1,w_2,w_3,w_4)}=\sum_{k=1}^4(-1)^kw_kx_k dx_1\wedge\ldots\wedge \widehat{dx_k}\wedge\ldots\wedge dx_4,
	\end{equation}
	where hat denotes the omission of the $k$-th factor. For hypersurfaces \eqref{eq:mirror}, $\Omega_{\mathbb{P}(w_1,w_2,w_3,w_4)}$ can be used to construct a basis of $H^3(\mathbb{P}(w_1,w_2,w_3,w_4)-X_b)$ by
	\begin{equation}
		\Xi_i=\frac{P_i\Omega_{\mathbb{P}(w_1,w_2,w_3,w_4)}}{f^k_{z_1,z_2}},
	\end{equation}
	where $P_i$ are homogeneous polynomials of degree $kd-(\sum w_j+1)$. $\Xi_i$ restrict to the hypersurface $X_b$ via the residue map $\mathrm{Res}: H^3(\mathbb{P}(w_1,w_2,w_3,w_4)-X_b)\rightarrow PH^2(X_b)$, where $PH^2(X_b)$ denotes the primitive cohomology of $X_b$. Let $\omega_j=\mathrm{Res}_{f_{z_1,z_2}(y)=0}(\Xi_j)\in PH^2(X_b)$. We define the Gauss-Manin connection by
	\begin{equation}
		\nabla_{\theta_i}\omega_j=\mathrm{Res}_{f_{z_1,z_2}(y)=0}(\theta_i\Xi_j),
	\end{equation}
	where $\theta_i=z_i\frac{\partial}{\partial z_i}$. It is straightforward to check that Griffiths transversality is satisfied and consequently a basis constructed from a fixed $(2,0)$ form $\omega_1$ by successive application of the Gauss-Manin connection is compatible with the Hodge filtration.
	The holomorphic $(2,0)$ form on $X_b$ can be chosen to be
	\begin{equation}
		\omega_1=\mathrm{Res}_{f_{z_1,z_2}(y)=0}\left(\sum_{k=1}^4(-1)^kw_ky_k\frac{dy_1\wedge\ldots\wedge \widehat{dy_k}\wedge\ldots\wedge dy_4}{f_{z_1,z_2}(y)}\right),
	\end{equation}
	We furthermore define 2-forms $\omega_i, i=2,3,4$ as
	\begin{equation}
		\omega_2=\nabla_{\theta_1}\omega_1,\quad \omega_3=\nabla_{\theta_2}\omega_1,\quad \mathrm{and}\quad
		\omega_4=\nabla_{\theta_1}\nabla_{\theta_1}\omega_1.
	\end{equation}
	For dimensional reasons $(\omega_1,\omega_2,\omega_3,\omega_4)$ indeed provide a basis for $H^2_{\mathrm{dR}}(X_b)_{\iota}$ and Griffiths transversality ensures its compatibility with the Hodge filtration.
	
	\begin{prop}\label{prop:GKZ}
		The Gauss-Manin connection in the basis $\omega=(\omega_1,\omega_2,\omega_3,\omega_4)$ is
		\begin{equation}
			\nabla_{\theta_i}\omega^{\rm tr}=G_i\omega^{\rm tr},\label{eq:GKZ}
		\end{equation}
		with
		\begin{equation}
			G_1=\begin{pmatrix}
				0&1&0&0\\
				0&0&0&1\\
				\frac{1}{2}\mu(1-\nu)z_1&\frac{1}{2}(\Delta_1-1)&0&\frac{1}{2}\Delta_1\\
				(G_1)_{41}&(G_1)_{42}&(G_1)_{43}&(G_1)_{44}
			\end{pmatrix},
		\end{equation}
		\begin{equation}
			G_2=\begin{pmatrix}
				0&0&1&0\\
				\frac{1}{2}\mu(1-\nu)z_1&\frac{1}{2}(\Delta_1-1)&0&\frac{1}{2}\Delta_1\\
				\frac{2\nu(1-\nu)z_1z_2}{\Delta_2}&\frac{(1-2\Delta_1)z_2}{\Delta_2}&\frac{2z_2}{\Delta_2}&\frac{(1-2\Delta_1)z_2}{\Delta_2}\\
				(G_2)_{41}&(G_2)_{42}&(G_2)_{43}&(G_2)_{44}
			\end{pmatrix},
		\end{equation}
		where we defined
		\begin{equation}
			\Delta_1=1-\mu\nu^2z_1,\qquad \Delta_2=1-4z_2,
		\end{equation}
		with $(\mu,\nu)$ given by $(3,3)$, $(4,4)$ and $(12,6)$ for K3 surfaces of elliptic type $E_6,~E_7$ and $E_8$ respectively. The entries $(G_i)_{4j}$ are collected in Appendix B. Parameters $\mu$ and $\nu$ can be computed from the toric data as
		\begin{equation}
			\mu=\frac{2d}{w_1w_2}\left(\frac{d}{w_1}\right)^{w_1/2-1}\left(\frac{d}{w_2}\right)^{w_2/2-1},\qquad\nu=\frac{d}{2}.
		\end{equation}
	\end{prop}
	
	\begin{proof}
		Picard-Fuchs system for hypersurfaces in weighted projective varieties is equivalent to the GKZ hypergeometric system \cite{Gel:1989} and can thus be determined from the generators of the Mori cone $l^{(i)}$ in Table \ref{tab:tab1}.
	\end{proof}
	
	\noindent Let $W_b$ be the Poincar\'e dual of $\iota(M)$ in $H^2_{\mathrm{dR}}(X_b)$. We define
	\begin{equation}
		H_2(X_b,\mathbb{Z})_\iota=H_2(X_b,\mathbb{Z})/W_b.
	\end{equation}
	Denote the basis of $H_2(X_b,\mathbb{Z})_\iota$ by $\gamma^j_b,~j=1,\ldots,m+2,$ and define the period matrix $\mathbf{\Pi}$ by integrating the basis $\omega$ of $H_{\mathrm{dR}}^2(X_b)_\iota$ over the integral cycles $\gamma^j_b\in H_2(X_b,\mathbb{Z})_{\iota}$
	\begin{equation}
		\mathbf{\Pi}_{ij}=\int_{\gamma^j_b }\omega_i,\quad \gamma^j_b\in H_2(X_b,\mathbb{Z})_\iota,\quad i,j=1,\ldots,m+2.
	\end{equation}
	The first row of $\mathbf{\Pi}$ corresponds to the periods $(X^0,X^1,X^2,X^3)$ of $\omega_1$. They satisfy the Picard-Fuchs differential equations
	\begin{equation}
		\begin{split}
			&\left(\theta_1(\theta_1-2\theta_2)-\mu z_1(\nu\theta_1+\nu-1)(\nu\theta_1+1)\right)\cdot X^j=0,\\
			&\left(\theta_2^2-z_2(\theta_1-2\theta_2)(\theta_1-2\theta_2-1)\right)\cdot X^j=0, \quad j=0,1,2,3,
		\end{split}\label{eq:pff}
	\end{equation}
	where $(\mu,\nu)$ are $(3,3)$, $(4,4)$ and $(12,6)$ as in Proposition \ref{prop:GKZ}. The system \eqref{eq:pff} admits a holomorphic solution
	\begin{equation}
		X^0=\sum_{n\geq2m\geq0}\frac{\left(\frac{d}{2}n\right)!}{\left(\frac{w_1}{2}n\right)!\left(\frac{w_2}{2}n\right)!(m!)^2(n-2m)!}z_1^nz_2^m,\label{eq:holp}
	\end{equation}
	with leading term 1. There are unique solutions $X^a,~a=1,2$ of \eqref{eq:pff} of the form
	\begin{equation}
		X^a=(2\pi i)^{-1}X^0\log(z_a)+S^a,\qquad a=1,2,
	\end{equation}
	where $S^a$ is a convergent power series in $z_1$ and $z_2$, with $S^a\rightarrow 0$ as $|z_i|\rightarrow 0$. The series $S^a$ are fixed uniquely as a solution to \eqref{eq:pff}.
	
	\begin{dfn}
		The Griffiths-Yukawa couplings $\mathsf{Y}_{ij}$ are
		\begin{equation}
			\mathsf{Y}_{ij}=-\int_X \omega_1\wedge \nabla_{\theta_i} \nabla_{\theta_j} \omega_1,\quad i,j=1,2.
		\end{equation}
	\end{dfn}
	\begin{prop}
		The Griffiths-Yukawa couplings for the mirrors of elliptically fibered K3 surfaces are given by
		\begin{equation}\label{{eq:yg}}
			\begin{split}
				&\mathsf{Y}_{11}=\frac{2c}{\Delta_1^2+(\Delta_2-1)(\Delta_1-1)^2},\\
				&\mathsf{Y}_{12}=\mathsf{Y}_{21}=\frac{c\Delta_1}{\Delta_1^2+(\Delta_2-1)(\Delta_1-1)^2},\\
				&\mathsf{Y}_{22}=\frac{2c(2\Delta_1-1)z_2}{\Delta_1^2+(\Delta_2-1)(\Delta_1-1)^2}.
			\end{split}
		\end{equation}
	\end{prop}
	\begin{proof}
		We compute
		\begin{equation}
			\theta_k \mathsf{Y}_{ij}=-\int_X\nabla_{\theta_k}\omega_1\wedge\nabla_{\theta_i}\nabla_{\theta_j}\omega_1-\int_X\omega_1\wedge\nabla_{\theta_i}\nabla_{\theta_j}\nabla_{\theta_k}\omega_1.
		\end{equation}
		Integrating by parts the first term, we find
		\begin{equation}\label{eq:eq23}
			\theta_k\mathsf{Y}_{ij}=-\frac{2}{3}\int_X\omega\wedge\nabla_{\theta_i}\nabla_{\theta_j}\nabla_{\theta_k}\omega_1.
		\end{equation}
		From the Gauss-Manin system we can express the action of $\nabla_{\theta_i}\nabla_{\theta_j}\nabla_{\theta_k}$ in terms of lower order operators. By Griffiths transversality only operators of second order will contribute and $\mathsf{Y}_{ij}$ are the solutions of the resulting differential equations. Moreover, we have the following relations between Griffiths-Yukawa couplings
		\begin{equation}
			\Delta_1 \mathsf{Y}_{11}-2\mathsf{Y}_{12}=0\qquad\mathrm{and}\qquad\Delta_2\mathsf{Y}_{22}+4z_2\mathsf{Y}_{12}-z_2\mathsf{Y}_{11}=0,
		\end{equation}
		which fix $\mathsf{Y}_{12}=\mathsf{Y}_{21}$ and $\mathsf{Y}_{22}$ in terms of $\mathsf{Y}_{11}$.
	\end{proof}
	\noindent The intersection pairing
	\begin{equation}
		Q:H^2_{\mathrm{dR}}(X)\times H^2_{\mathrm{dR}}(X)\rightarrow \mathbb{C}, \quad Q(\omega_i,\omega_j)=\int_X\omega_i\wedge\omega_j,
	\end{equation}
	in the basis $\omega$ will be denoted by $Q_\omega$. It is given by
	\begin{equation}
		Q_\omega=\begin{pmatrix}
			0&0&0&-\mathsf{Y}_{11}\\
			0&\mathsf{Y}_{11}&\mathsf{Y}_{12}&\frac{1}{2}\theta_1\mathsf{Y}_{11}\\
			0&\mathsf{Y}_{21}&\mathsf{Y}_{22}&-\frac{1}{2}\theta_2\mathsf{Y}_{11}+\theta_1\mathsf{Y}_{12}\\
			-\mathsf{Y}_{11}&\frac{1}{2}\theta_1\mathsf{Y}_{11}&-\frac{1}{2}\theta_2\mathsf{Y}_{11}+\theta_1\mathsf{Y}_{12}&\mathsf{Y}_{44}
		\end{pmatrix},
	\end{equation}
	with
	\begin{equation}
		\begin{split}
			\mathsf{Y}_{44}&=-\frac{1}{\Delta_1^2+(\Delta_2-1)(\Delta_1-1)^2}\left(-4\theta_1^2\mathsf{Y}_{11}+\frac{1}{2}(\Delta_1-1)(1+\Delta_2(1+\Delta_1))\theta_1\mathsf{Y}_{11}\right.\\
			&\left.+\frac{1}{2}\left((\Delta_1-1)(2\Delta_2(1-\Delta_1)-1)+4\mu(\nu-1)z_1(4(1-\Delta_2)+3\Delta_2)\right)\mathsf{Y}_{11}\right).
		\end{split}
	\end{equation}
	Note that the logarithmic derivatives of $\mathsf{Y}_{ij}$ can be expressed in terms of multiplication factors
	\begin{equation}
		\theta_1\mathsf{Y}_{11}=\frac{(1-\Delta_1)(1+(\Delta_1-1)\Delta_2)}{\Delta_1^2+(\Delta_2-1)(\Delta_1-1)^2}\mathsf{Y}_{11},\qquad \theta_2\mathsf{Y}_{11}=\frac{(1-\Delta_1)^2(\Delta_2-1)}{\Delta_1^2+(\Delta_2-1)(\Delta_1-1)^2}\mathsf{Y}_{11},
	\end{equation}
	hence all the entries in $Q_\omega$ can be expressed in terms of $\mathsf{Y}_{11}$ and algebraic prefactors.

	\section{Algebraic variation of Hodge structure and differential rings}\label{sec:vhs}
	\subsection{Moduli space of enhanced K3 surfaces}
	By moduli spaces $\mathsf{T}$ of lattice polarized K3 surfaces enhanced with differential forms we refer to the moduli spaces of pairs $(X,\{\omega_i\}_{i=1,\ldots,m+2})$ where $X$ is a K3 surface polarized by a lattice $M$ of rank $\mathrm{rk}(M)=20-m$ with signature $(1,19-m)$ and $\{\omega_i\}_{1=1,\ldots,m+2}$ is a basis of $H_{\mathrm{dR}}^2(X)/\iota(M)$, where $\iota:M\rightarrow H_{\mathrm{dR}}^2(X)$ denotes the polarization map. A basis $\{\alpha_i\}_{i=1,\ldots,m+2}$ of $H_{\mathrm{dR}}^2(X)/\iota(M)$ can be fixed, such that the intersection pairing in this basis is given by the pairing
	\begin{equation}
		\Phi=\begin{pmatrix}
			0&0&-1\\
			0&C_{ab}&0\\
			-1&0&0
		\end{pmatrix},
	\end{equation}
	where $C_{ab}$ denote the intersection numbers of the elliptic K3 surface. Mirror families introduced in the previous section are two-parameter families of hypersurfaces in weighted projective spaces \eqref{eq:mirror}. They are polarized by the pull-back of the lattice of toric divisors to the hypersurface. The rank of $\mathrm{Pic}_{cor}(X^\circ)$ is $18$, $m=2$ and the moduli space $\mathsf{T}$ is 6-dimensional. Away from the discriminant locus $\mathsf{T}$ is a locally ringed space with the local ring $\mathcal{O}_\mathsf{T}$.  We will show that there is an isomorphism
	\begin{equation}
		\mathcal{O}_{\mathsf{T}}\cong \widetilde{\mathsf{M}}(\Gamma_0(N)\times\Gamma_0(N)),\label{eq:iso2}
	\end{equation}
	between the local ring $\mathcal{O}_{\mathsf{T}}$ and the graded ring of quasi-modular forms of the modular subgroup $\Gamma_0(N)$ in two variables. The level $N$ of the congruence subgroup is determined by the type of elliptic fiber of the elliptic K3, as explained in the next subsection. For explicit construction of the coordinates on $\mathsf{T}$ consider the filtration preserving transformation $\omega\mapsto \alpha=\mathsf{S}\omega,~G_i\mapsto G_a= \sum_i\frac{1}{z_i}\frac{\partial z_i}{\partial t_a}\mathsf{S}G_i\mathsf{S}^{-1}+\partial_a \mathsf{S}\cdot\mathsf{S}^{-1},~i,a=1,2$, where $\mathsf{S}$ is of the form
	\begin{equation}
		\mathsf{S}=\begin{pmatrix}
			\mathsf{s}_0&0&0\\
			\mathsf{s}_a&\mathsf{s}_{a,i}&0\\
			\mathsf{s}_{3,0}&\mathsf{s}_{3,i}&\mathsf{s}_{3,3}
		\end{pmatrix},
	\end{equation}
	and $\partial_a=\frac{\partial}{\partial t_a}$ denotes the differentiation with respect to coordinates $t_a$. The moduli space $\mathsf{T}$ of K3 surfaces enhanced with differential forms consists of the moduli space $\mathsf{B}$, together with the independent parameters of $\mathsf{S}$. The condition on the pairing
	\begin{equation}
		\mathsf{S}Q_\omega \mathsf{S}^{\mathrm{tr}}=\Phi,
	\end{equation}
	reads explicitly:
	\begin{equation}
		\begin{split}
			\mathsf{s}_{3,3}&=\frac{1}{\mathsf{s}_0\mathsf{Y}_{1,1}},\\
			\mathsf{C}_{ab}^{alg}&=\mathsf{s}_{a,i}\mathsf{s}_{b,j}\mathsf{Y}_{ij},\quad i,j=1,2,\\
			s_{3,i}&=\frac{1}{\mathsf{s}_0}\mathsf{s}^{-1}_{a,j}\mathsf{Y}^{-1}_{ji}\mathsf{s}_{a}+\frac{1}{\mathsf{s}_0}\frac{\mathsf{Y}^{-1}_{i1}\mathsf{Y}_{24}+\mathsf{Y}^{-1}_{i2}\mathsf{Y}_{34}}{\mathsf{Y}_{11}},\\
			\mathsf{s}_{3,0}&=\frac{1}{2\mathsf{s}_0}(\mathsf{C}^{alg})^{-1}_{ab}\mathsf{s}_{a,0}\mathsf{s}_{b,0}+\frac{\mathsf{Y}_{22}(\mathsf{Y}_{24}^2+\mathsf{Y}_{11}\mathsf{Y}_{44})-\mathsf{Y}_{12}^2\mathsf{Y}_{44}-2\mathsf{Y}_{12}\mathsf{Y}_{24}\mathsf{Y}_{34}}{2\mathsf{s}_0\mathsf{Y}_{11}^2(\mathsf{Y}_{11}\mathsf{Y}_{22}-\mathsf{Y}_{12}^2)}.
		\end{split}
	\end{equation}
	For elliptic K3 surfaces we find 4 independent parameters. As coordinates on $\mathsf{T}$ we choose $\mathsf{s}_0, \mathsf{s}_1, \mathsf{s}_2$ and $\mathsf{s}_{1,1}$.
	
	\subsection{Algebraic variation of Hodge structure for projective elliptic K3 surfaces}

	Let $(\mathcal{X},\alpha)\rightarrow \mathsf{T}$ be a family of lattice polarized projective K3 surfaces with a fixed choice of basis $\alpha$ of $\mathcal{H}^2_{\mathrm{dR}}(\mathcal{X})_\iota$, such that the intersection pairing in the basis $\alpha$ is $\Phi$. Let furthermore $\mathsf{R}$ denote the function ring of $\mathsf{T}$. The relative algebraic de Rham cohomology $H^2_{\mathrm{dR}}(\mathcal{X}/\mathsf{T})$ (see \cite{Grothendieck:1966}) carries the Gauss-Manin connection $\nabla: H^2_{\mathrm{dR}}(\mathcal{X})\rightarrow H^2_{\mathrm{dR}}(\mathcal{X}/\mathsf{T})\otimes_{\mathsf{R}}\Omega^1_{\mathsf{T}}$, where $\Omega^1_{\mathsf{T}}$ is an $\mathsf{R}$-module of differential forms in $\mathsf{R}$ \cite{Katz:1968}. As before, the Gauss-Manin connection restricts to the quotient $\mathcal{H}^2_{\mathrm{dR}}(\mathcal{X})_\iota$. Let $\mathrm{Vec}(\mathsf{T})$ be the Lie algebra of vector fields on $\mathsf{T}$. The algebraic Gauss-Manin connection $\nabla$ acts on $\alpha$ as
	\begin{equation}
		\nabla_{\mathsf{E}_i}\alpha=\mathsf{A}_{\mathsf{E}_i}\alpha,\qquad \mathsf{E}_i\in \mathrm{Vec}(\mathsf{T}),
	\end{equation}
	where $\mathsf{A}_{\mathsf{E}_i}$ are $(m+2)\times (m+2)$ matrices with entries in $\mathcal{O}_{\mathsf{T}}$.
	
	\begin{thm}\label{thm:mvf}
		There are unique vector fields $\mathsf{R}_a\in\mathrm{Vec}(\mathsf{T})$ and unique $\mathsf{C}_{ab}^{alg}\in\mathcal{O}_{\mathsf{T}}, a,b=1,2$ symmetric in $a,b$ such that
		\begin{equation}
			\mathsf{A}_{\mathsf{R}_a}=\begin{pmatrix}
				0&\delta^b_a&0\\
				0&0&\mathsf{C}_{ac}^{alg}\\
				0&0&0
			\end{pmatrix}.\label{eq:modvf}
		\end{equation}
		We call them modular vector fields. Furthermore
		\begin{equation}
			\mathsf{R}_{a_1}\mathsf{C}_{a_2a_3}^{alg}=0.
		\end{equation}
	\end{thm}
	Theorem \ref{thm:mvf} amounts to finding $\mathsf{S}$ and $t_a$ as above such that
	\begin{equation}
		\mathsf{A}_{\frac{\partial}{\partial t_a}}=\sum_i\frac{1}{z_i}\frac{\partial z_i}{\partial t_a}\mathsf{S}G_i \mathsf{S}^{-1}+\partial_a \mathsf{S}\cdot\mathsf{S}^{-1}.\label{eq:flat}
	\end{equation}
	\begin{prop}
		The system \eqref{eq:flat} is solved by
		\begin{equation}
			\mathsf{s}_0=\left(X^0\right)^{-1},\qquad t_a=\frac{X^a}{X^0},\label{eq:isoms}
		\end{equation}
		where $X^0$ denotes the fundamental period and $X^a$ denote the the periods with a logarithmic pole at $z_a=0$. Furthermore, the other independent parameters in $\mathsf{S}$ satisfy
		\begin{equation}\label{eq:re}
			\mathsf{s}_{a,i}=\frac{1}{z_i}\frac{\partial z_i}{\partial t_a}\mathsf{s}_0,\quad
			\mathsf{s}_a=\sum_{i=1,2}\mathsf{s}_{a,i}\theta_i \log\mathsf{s}_0=\partial_{t_a}\log \mathsf{s}_0.
		\end{equation}
	\end{prop}
	\begin{proof}
		The proof is computational. For a proof using the special K\"ahler structure of $\mathsf{B}$ see \cite{Alim:2014}.
	\end{proof}
	\begin{cor}
		With this choice, we find that $\mathsf{C}_{ab}^{alg}$ are $\mathsf{C}_{11}^{alg}=C_{HH}, \mathsf{C}_{12}^{alg}=\mathsf{C}_{21}^{alg}=C_{HL}$ and $\mathsf{C}_{22}^{alg}=C_{LL}=0$, which finishes the proof of Theorem \ref{thm:mvf}.
	\end{cor}
	
	\noindent We define weight 1 quasi-modular forms for genus zero congruence subgroups in Appendix A.
	\begin{thm}\label{thm:LY}
		\cite{Lian:19951,Lian:19952} The fundamental period $X^0$ of a mirror to a projective elliptic K3 surface factorizes
		\begin{equation}\label{eq:fac}
			X^0=A(\tau_1)A(\tau_2),
		\end{equation}
		with $\tau_1=t_1, \tau_2=t_1+t_2$ and the weight 1 modular forms $A$ are given in Appendix A. For each model, the quasi-modular form $A$ is the quasi-modular form associated to the monodromy group of the respective elliptic fiber. This can be checked by comparison with \eqref{eq:holp}.
	\end{thm}
	
	\begin{prop}
		There is an isomorphism
		\begin{equation}
			\mathcal{O}_{\mathsf{T}}\cong \widetilde{\mathsf{M}}(\Gamma_0(N)\times\Gamma_0(N)),\label{eq:iso3}
		\end{equation}
		between the local ring $\mathcal{O}_{\mathsf{T}}$ and the graded ring of quasi-modular forms of the modular subgroup $\Gamma_0(N)$ in two variables. The level $N$ of the congruence subgroup is the same as in the monodromy group of the elliptic fibre of the elliptic K3 surface.
	\end{prop}
	
	\begin{proof}
		The independent variables $\mathsf{t}:=\{z_i,\mathsf{s}_0,\mathsf{s}_a,\mathsf{s}_{1,1}\}_{i,a=1,2}$ form a local chart for $\mathsf{T}$. Denote by $\mathcal{O}_\mathsf{T}$ the local ring at $\mathsf{t}\in\mathsf{T}$. Theorem \ref{thm:LY} provides an isomorphism between the ring $\mathcal{O}_\mathsf{T}$ and the ring of quasi-modular forms in two variables. The isomorphism is given by the inverse mirror map $z_i=z_i(t_1,t_2)$. Fix $\alpha_a=\left(\frac{C(\tau_a)}{A(\tau_a)}\right)^r$, $r$ as in Appendix A. It satisfies
		\begin{equation}
			\partial_{\tau_a}\alpha_b=\delta_a^b\alpha_b(1-\alpha_b)A^2(\tau_b).
		\end{equation}
		In terms of these the inverse mirror map is
		\begin{equation}
			\begin{split}
				z_1=\frac{1}{d_N}(\alpha_1+\alpha_2-2\alpha_1\alpha_2),\qquad z_2=\frac{1}{d_N^2}\frac{\alpha_1\alpha_2(1-\alpha_1)(1-\alpha_2)}{z_1^2}.
			\end{split}
		\end{equation}
		The remaining elements of the ring are found from \eqref{eq:isoms}, \eqref{eq:re}, \eqref{eq:fac} and \eqref{eq:ring}
		\begin{equation}\begin{split}
				s_a&=-\frac{1}{2r}\left(E(\tau_a)+\frac{2C^r(\tau_a)-A^r(\tau_a)}{A^{r-2}(\tau_a)}\right),\\
				\mathsf{s}_{1,1}&=\frac{\alpha_1(1-\alpha_1)(1-2\alpha_2)}{\alpha_1(1-\alpha_2)+\alpha_2(1-\alpha_1)}\frac{A(\tau_1)}{A(\tau_2)}.
			\end{split}
		\end{equation}
	\end{proof}
	
	\section{The Gauss-Manin Lie algebra}
	\subsection{Algebraic Group acting on $\mathsf{T}$}
	We define a Lie group $\mathsf{G}$ by
	\begin{equation}
		\mathsf{G}=\{\mathsf{g}\in GL(m+2,\mathbb{C})|~ \mathsf{g}~\mathrm{block~lower~triangular~and}~\mathsf{g}\Phi\mathsf{g}^{\rm tr}=\Phi\}.\label{eq:lg}
	\end{equation}
	It acts on $\mathsf{T}$ from the right as
	\begin{equation}
		(X,\alpha)\mathbf{\cdot} \mathsf{g}=(X,\alpha^{\rm tr}\mathsf{g}),
	\end{equation}
	where $\alpha=(\alpha_1,\ldots,\alpha_{m+2})^{\rm tr}$ is the special basis defined in the section \ref{sec:ms}, $ \mathsf{g}\in\mathsf{G}$, and $\alpha^{\rm tr}\mathsf{g}$ is the standard matrix product. The condition $\mathsf{g}\Phi\mathsf{g}^{\rm tr}=\Phi$ fixes $\dim(\mathsf{G})=\dim(\mathsf{T})-2=4$. The group $\mathsf{G}$ is generated by two elements isomorphic to the multiplicative group $\mathbb{C}^*$ and two elements isomorphic to the additive group $\mathbb{C}$. The following lemma gives the generators of $\mathsf{G}$:
	\begin{lem}
		For any $\mathsf{g}\in\mathsf{G}$ there are unique elements $\mathsf{g}_i\in\mathsf{G},~i=1,2,3,4$ such that $\mathsf{g}$ can be written as a product of at most four $\mathsf{g}_i$. For families \eqref{eq:mirror}, $\mathsf{g}_i$ are given by
		\begin{equation}
			\begin{split}
				\mathsf{g}_1&=\begin{pmatrix}
					\mathsf{h}_0&0&0&0\\0&1&0&0\\0&0&1&0\\0&0&0&\mathsf{h}_0^{-1}
				\end{pmatrix},\quad \mathsf{g}_2=\begin{pmatrix}
					1&0&0&0\\ 0&\mathsf{h}_{11}&\mathsf{h}_{21}&0\\ 0&\mathsf{h}_{12}&\mathsf{h}_{22}&0\\0&0&0&1
				\end{pmatrix},\\
				\mathsf{g}_3&=\begin{pmatrix}
					1&0&0&0\\ \mathsf{C}_{11}^{alg}\mathsf{h}_1&1&0&0\\ \mathsf{C}_{12}^{alg}\mathsf{h}_1&0&1&0\\0&\mathsf{h}_1&0&1
				\end{pmatrix},\quad \mathsf{g}_4=\begin{pmatrix}
					1&0&0&0\\ \mathsf{C}_{12}^{alg}\mathsf{h}_2&1&0&0\\ \mathsf{C}_{22}^{alg}\mathsf{h}_2&0&1&0\\0&0&\mathsf{h}_2&1
				\end{pmatrix},
			\end{split}
		\end{equation}
		where $\mathsf{h}_{ij},~i,j=1,2$ satisfy the constraints
		\begin{equation}
			\sum_{i,j=1,2}\mathsf{C}_{ij}^{alg}\mathsf{h}_{ik}\mathsf{h}_{jl}=\mathsf{C}_{lk}^{alg}.
		\end{equation}
		The constraints can be solved by a simple algebraic manipulation, which yields only one independent parameter. We fix $\mathsf{h}_{11}=\mathsf{C}_{12}^{alg}\mathsf{h}_3$ as the independent parameter and express $\mathsf{h}_{ij}$ in terms of $\mathsf{h}_3$.
	\end{lem}
	
	\noindent The Lie algebra of $\mathsf{G}$ is given by
	\begin{equation}
		\mathrm{Lie}(\mathsf{G})=\{\mathfrak{g}\in\mathrm{Mat}(m+2,\mathbb{C})|~\mathfrak{g}~\mathrm{is~block~lower~triangular~and}~\mathfrak{g}\Phi+\Phi\mathfrak{g}=0\}.\label{eq:la}
	\end{equation}
	The Lie algebra $\mathrm{Lie}(\mathsf{G})$ is a Lie sub-algebra of $\mathrm{Vec}(\mathsf{T})$. The basis of $\mathrm{Lie}(\mathsf{G})$ can be constructed from the elements $\mathsf{g}_i\in\mathsf{G}$. We find
	\begin{equation}
		\begin{split}
			\mathfrak{g}_1&=\begin{pmatrix}
				1&0&0&0\\0&0&0&0\\0&0&0&0\\0&0&0&-1
			\end{pmatrix},\quad \mathfrak{g}_2=\begin{pmatrix}
				0&0&0&0\\ 0&\mathsf{C}_{12}^{alg}&-\mathsf{C}_{11}^{alg}&0\\ 0&\mathsf{C}_{22}^{alg}&-\mathsf{C}_{12}^{alg}&0\\0&0&0&0
			\end{pmatrix},\\
			\mathfrak{g}_3&=\begin{pmatrix}
				0&0&0&0\\ \mathsf{C}_{11}^{alg}&0&0&0\\ \mathsf{C}_{12}^{alg}&0&0&0\\0&1&0&0
			\end{pmatrix},\quad\mathfrak{g}_4=\begin{pmatrix}
				0&0&0&0\\ \mathsf{C}_{12}^{alg}&0&0&0\\ \mathsf{C}_{22}^{alg}&0&0&0\\0&0&1&0
			\end{pmatrix}.
		\end{split}
	\end{equation} 
	
	\subsection{The Gauss-Manin Lie algebra}
	The Gauss-Manin Lie algebra is defined to be the $\mathcal{O}_{\mathsf{T}}$ module generated by $\mathrm{Lie}(\mathsf{G})$ and the modular vector fields $\mathsf{R}_a$ in \eqref{eq:modvf}. We write the action of the modular vector fields $\mathsf{R}_a$ on $\alpha$ explicitly
	\begin{equation}
		\mathsf{A}_{\mathsf{R}_1}=\begin{pmatrix}
			0&1&0&0\\
			0&0&0&\mathsf{C}_{11}^{alg}\\
			0&0&0&\mathsf{C}_{12}^{alg}\\
			0&0&0&0
		\end{pmatrix},\qquad \mathsf{A}_{\mathsf{R}_2}=\begin{pmatrix}
			0&0&1&0\\
			0&0&0&\mathsf{C}_{12}^{alg}\\
			0&0&0&\mathsf{C}_{22}^{alg}\\
			0&0&0&0
		\end{pmatrix}.
	\end{equation}
	\begin{thm}\label{thm:mainthm}
		The Gauss-Manin Lie algebra, generated by $\mathfrak{g}_1, \mathfrak{g}_2, \mathfrak{g}_3, \mathfrak{g}_4$ and $\mathsf{A}_{\mathsf{R}_1}, \mathsf{A}_{\mathsf{R}_2}$, is isomorphic to $\mathrm{sl}_2(\mathbb{C})\oplus\mathrm{sl}_2(\mathbb{C})$.
	\end{thm}
	\begin{proof}
		Let
		\begin{equation}
			A=\begin{pmatrix}
				1&0&0&0\\
				0&\frac{w_1w_2}{2d}&-2\frac{w_1w_2}{2d}&0\\
				0&0&\frac{w_1w_2}{2d}&0\\
				0&0&0&1
			\end{pmatrix}.
		\end{equation}
		Here $\frac{w_1w_2}{2d}=C_{HL}^{-1}$, as in Section \ref{sec:ms}. The Lie algebra is given by the generators
		\begin{equation}
			\begin{split}
				\mathcal{J}^1=A\cdot (\mathfrak{g}_1+\mathfrak{g}_2)&=\begin{pmatrix}
					1&0&0&0\\0&1&0&0\\0&0&-1&0\\0&0&0&-1
				\end{pmatrix},\quad \mathcal{J}^2=A\cdot (\mathfrak{g}_1-\mathfrak{g}_2)=\begin{pmatrix}
					1&0&0&0\\0&-1&0&0\\0&0&1&0\\0&0&0&-1
				\end{pmatrix},\\
				\mathcal{J}^1_-=A\cdot \mathfrak{g}_3&=\begin{pmatrix}
					0&0&0&0\\0&0&0&0\\1&0&0&0\\0&1&0&0
				\end{pmatrix},\quad \mathcal{J}^2_-=A\cdot \mathfrak{g}_4=\begin{pmatrix}
					0&0&0&0\\1&0&0&0\\0&0&0&0\\0&0&1&0
				\end{pmatrix},\\
				\mathcal{J}^1_+=A\cdot \mathsf{A}_{\mathsf{R}_2}&=\begin{pmatrix}
					0&0&1&0\\0&0&0&1\\0&0&0&0\\0&0&0&0
				\end{pmatrix}\quad \mathcal{J}^2_+=A\cdot \mathsf{A}_{\mathsf{R}_1}=\begin{pmatrix}
					0&1&0&0\\0&0&0&0\\0&0&0&1\\0&0&0&0
				\end{pmatrix}.
			\end{split}
		\end{equation}
		which form a basis of $\mathrm{sl}_2(\mathbb{C})\oplus\mathrm{sl}_2(\mathbb{C})$, with commutation relations
		\begin{equation}
			[\mathcal{J}^a_+,\mathcal{J}^a_-]=\mathcal{J}^a,\quad [\mathcal{J}^a_0,\mathcal{J}^a_+]=\mathcal{J}^a_+,\quad
			[\mathcal{J}^a_0,\mathcal{J}^a_-]=-\mathcal{J}^a_-,\quad
			[\mathcal{J}^1_\bullet,\mathcal{J}^2_\bullet]=0,\quad a=1,2,
		\end{equation}
		where $\bullet$ denotes any generator.
	\end{proof}
	
	\section{Conclusions}
	
	Movasati's work on the enhanced moduli space of elliptic curves \cite{Movasati:20113} provided the VHS and algebraic context for quasi-modular forms with the corresponding $\mathrm{sl}_2(\mathbb{C})$ Lie algebra. The general automorphic or modular properties for moduli spaces of CY threefolds such as the quintic are less clear (see Ref.~\cite{Movasati:2017}), although the analogous enhanced moduli spaces and Gauss-Manin Lie algebras have been put forward. Lattice polarized K3 manifolds therefore provide the middle grounds between the classical theory of quasi modular forms and new structures appearing in the moduli spaces of generic threefolds. We should note, that in various limiting constructions both of elliptically fibered CY manifolds as well as non-compact CY manifolds, connections to classical quasi modular forms have been worked out, see Ref.~\cite{Haghighat:2017} for a review. Correspondingly $\mathrm{sl}_2(\mathbb{C})$ Lie-subalgebras of the full Gauss-Manin (GMCD) Lie algebras have been put forward in the context of elliptic fibrations \cite{Haghighat:2015}. A different kind of universal $\mathrm{sl}_2(\mathbb{C})$
	Lie sub-algebra of the Gauss-Manin Lie algebra stemming from the rescaling of the holomorphic top form of the Griffiths-Dwork family of CY $d$-folds has been studied in Ref.~\cite{Nikdelan:2017}. Our result provides on the other hand the full Gauss-Manin Lie algebra which is beyond, yet very close to the classical one, reducing to a direct sum of two copies the classical one. It would be interesting to investigate the Gauss-Manin Lie algebra in detail for more intricate geometries of lattice polarized K3 manifolds and perhaps obtain a complete classification of the Lie algebras that can be obtained in this way. The description of the Gauss-Manin Lie algebras in terms of the data of the intersection form which is needed for that was already given in Ref.~\cite{Alim:2014}. We note moreover that we expect the study of mirror families with three-dimensional moduli spaces of complex structures to provide the arena for the study of the analogs of quasi-modular forms for Siegel modular forms as alluded to in Ref.~\cite{Doran:2014}. A further possible line of investigation would be to understand the relevance of the enhanced moduli space $\mathsf{T}$ for the enumerative geometry of the Calabi-Yau threefolds that admit a K3 fibration, as demonstrated in Ref.~\cite{Oberdieck:2017}. In a similar way the differential ring of $T^2$ was exploited in studying the threefold K3$\times T^2$, which lead to a novel computation of certain BPS degeneracies in Ref.~\cite{Kachru:2017} and the proof of Igusa cusp form conjecture in Ref.~\cite{Oberdieck:2018}. 
	
	
	\appendix
	
	\section{Ramanujan-Serre differential ring for congruence subgroups}
	Let $\mathbb{H}=\left\{\tau~|~\mathrm{Im}(\tau)>0\right\}$ be the upper half plane and the group $SL_2(\mathbb{Z})$ the group of matrices $\begin{pmatrix}a&b\\c&d\end{pmatrix}$ with integer entries satisfying $ad-bc=1$. The genus zero congruence subgroups of $SL_2(\mathbb{Z})$ are defined as
	\begin{equation}
		\Gamma_0(N)=\left\{\begin{pmatrix}a&b\\c&d\end{pmatrix}\in SL_2(\mathbb{Z})|c\equiv0 ~\mathrm{mod}~N\right\}.
	\end{equation}
	We define weight one modular forms associated to the congruence subgroups:
	\begin{table}[h]\label{tab:tab3}
		\centering
		\begin{tabular}{c|c  c  c}
			$N$&$A$&$B$&$C$\\
			1&$E_4(\tau)^{1/4}$&$\left(\frac{E_4(\tau)^{3/2}+E_6(\tau)}{2}\right)^{1/6}$&$\left(\frac{E_4(\tau)^{3/2}-E_6(\tau)}{2}\right)^{1/6}$\\
			2&$\frac{(64\eta(2\tau)^{24}+\eta(\tau)^{24})^{1/4}}{\eta(\tau)^2\eta(2\tau)^2}$&$\frac{\eta(\tau)^4}{\eta(2\tau)^2}$&$2^{3/2}\frac{\eta(2\tau)^4}{\eta(\tau)^2}$\\
			3&$\frac{(27\eta(3\tau)^{12}+\eta(\tau)^{12})^{1/3}}{\eta(\tau)\eta(3\tau)}$&$\frac{\eta(\tau)^3}{\eta(3\tau)}$&$3\frac{\eta(3\tau)^3}{\eta(\tau)}$
		\end{tabular}
	\end{table}\\
	\noindent Where $E_4(\tau)$ and $E_6(\tau)$ denote the Eisenstein series and $\eta(\tau)$ denotes the Dedekind $\eta$-function. Define also the analogue of the Eisenstein series $E_2$
	\begin{equation}\label{eq:ee}
		E=\partial_\tau\log B^rC^r,
	\end{equation}
	where $r=6$ for $N=1$, $r=4$ for $N=2$ and $r=3$ for $N=3$. The ring of quasi-modular forms for genus zero congruence subgroups is $\widetilde{\mathsf{M}}(\Gamma_0(N))=\mathbb{C}[A,B,E]$.
	
	From the Ramanujan-Serre differential ring of Eisenstein series we deduce the following differential relations
	\begin{equation}\label{eq:ring}
		\begin{split}
			\partial_\tau A&=\frac{1}{2r}A\left(E+\frac{C^r-B^r}{A^{r-2}}\right),\\
			\partial_\tau B&=\frac{1}{2r}B(E-A^2),\\
			\partial_\tau C&=\frac{1}{2r}C(E+A^2),\\
			\partial_\tau E&=\frac{1}{2r}(E^2-A^4).
		\end{split}
	\end{equation}
	
	The $j$-function for genus zero congruence subgroups $\Gamma_0(N)$ reads
	\begin{equation}\label{eq:jfctn}
		j=\frac{d_NA^{2r}}{C^r(A^r-C^r)},
	\end{equation}
	where $d_N=432$ for $\Gamma_0(1)$, $d_N=64$ for $\Gamma_0(2)$ and $d_N=27$ for $\Gamma_0(3)$. Note that $d_N=\mu\nu^2$ in \eqref{eq:pff}, which is to be expected, as in the limit $z_2\rightarrow0$ the system reduces to the Picard-Fuchs equation of the elliptic curve.
	
	\section{Gauss-Manin connection matrices}
	The entries $(G_i)_{4j}$ of the Gauss-Manin connection matrices are
	\begin{equation}
		\begin{split}
			(G_1)_{41}&=\frac{\mu(1-\nu)z_1(2(1-\Delta_1)-1)}{\Delta_1^2+(\Delta_2-1)(\Delta_1-1)^2},\\
			(G_1)_{42}&=\frac{\mu z_1((1-\nu)((2-\Delta_1)\Delta_2-2)+\nu^2(2\Delta_1\Delta_2-1))}{\Delta_1^2+(\Delta_2-1)(\Delta_1-1)^2},\\
			(G_1)_{43}&=\frac{2\mu(1-\nu)z_1\Delta_2}{\Delta_1^2+(\Delta_2-1)(\Delta_1-1)^2},\\
			(G_1)_{44}&=\frac{3\mu\nu^2z_1(1-(1-\Delta_1)\Delta_2)}{\Delta_1^2+(\Delta_2-1)(\Delta_1-1)^2},
		\end{split}
	\end{equation}
	and
	\begin{equation}
		\begin{split}
			(G_2)_{41}&=\frac{\mu(1-\nu)(1-\Delta_1)z_1(1-(1+\Delta_1)\Delta_2)}{2(\Delta_1^2+(\Delta_2-1)(\Delta_1-1)^2)},\\
			(G_2)_{42}&=\frac{\mu z_1((1-\nu)(1-\Delta_2)-\nu^2(1-\Delta_1)(1-(1+\Delta_1)\Delta_2)}{2(\Delta_1^2+(\Delta_2-1)(\Delta_1-1)^2)},\\
			(G_2)_{43}&=-\frac{\mu(1-\nu)z_1\Delta_1\Delta_2}{\Delta_1^2+(\Delta_2-1)(\Delta_1-1)^2},\\
			(G_2)_{44}&=\frac{(1-\Delta_1)(-1+(1-\Delta_1)(3\Delta_2-1)-(1-\Delta_1^2)\Delta_2)}{2(\Delta_1^2+(\Delta_2-1)(\Delta_1-1)^2)}.
		\end{split}
	\end{equation}

	\small
\newcommand{\etalchar}[1]{$^{#1}$}

\end{document}